\documentclass[12pt]{amsart}
\usepackage{fullpage,url,amssymb}

\DeclareFontEncoding{OT2}{}{} 

\usepackage{color}

\newcommand{\defi}[1]{\textsf{#1}} 

\newcommand{\Aff}{{\mathbb A}}
\newcommand{\C}{{\mathbb C}}

\newcommand{\PP}{{\mathbb P}}
\newcommand{\Q}{{\mathbb Q}}

\newcommand{\Z}{{\mathbb Z}}
\newcommand{\Qbar}{{\overline{\Q}}}

\newcommand{\calI}{{\mathcal I}}

\newcommand{\OO}{{\mathcal O}}

\DeclareMathOperator{\End}{End}

\DeclareMathOperator{\Aut}{Aut}

\DeclareMathOperator{\Spec}{Spec}

\newcommand{\tH}{{\operatorname{th}}}

\newcommand{\GL}{\operatorname{GL}}

\newcommand{\PGL}{\operatorname{PGL}}

\newcommand{\injects}{\hookrightarrow}
\newcommand{\isom}{\simeq}

\newcommand{\del}{\partial}
\newcommand{\intersect}{\cap} 

\newtheorem{theorem}{Theorem}[section]
\newtheorem{lemma}[theorem]{Lemma}
\newtheorem{corollary}[theorem]{Corollary}

\theoremstyle{definition}

\theoremstyle{remark}
\newtheorem{remark}[theorem]{Remark}

\usepackage[
	backref,
	pdfauthor={Bjorn Poonen}, 
]{hyperref}
\usepackage[alphabetic,backrefs,lite]{amsrefs} 

\begin{document}

\title{Automorphisms mapping a point into a subvariety}
\subjclass[2000]{Primary 14Q20; Secondary 11U05}
\keywords{automorphism, Hilbert's tenth problem, undecidability}
\author{Bjorn Poonen}
\thanks{This research was supported by NSF grant DMS-0841321.}
\address{Department of Mathematics, Massachusetts Institute of Technology, Cambridge, MA 02139-4307, USA}
\email{poonen@math.mit.edu}
\urladdr{http://math.mit.edu/~poonen}
\date{February 20, 2009}

\begin{abstract}
The problem of deciding, given a complex variety $X$,
a point $x \in X$, and a subvariety $Z \subseteq X$,
whether there is an automorphism of $X$ mapping $x$ into $Z$
is proved undecidable.
Along the way, we prove the undecidability of a version
of Hilbert's tenth problem for systems of polynomials over $\Z$
defining an affine $\Q$-variety whose projective closure is smooth.
\end{abstract}

\maketitle

\section{Introduction}\label{S:introduction}

\begin{theorem}
\label{T:main}
There is no algorithm that, given a
nice complex variety $X$,
a closed point $x \in X$, and a nice subvariety $Z \subseteq X$,
decides whether or not there is an automorphism of $X$
mapping $x$ into $Z$.
\end{theorem}
\noindent
\defi{Variety} means separated scheme of finite type over a field.
\defi{Nice} means smooth, projective, and geometrically integral
(we will eventually apply this adjective 
also to varieties over fields that are not algebraically closed).
\defi{Algorithm} means Turing machine.
So that the input admits a finite description, 
we assume that the input includes a description of a finitely
generated subfield $K$ of $\C$ and that 
the coefficients of the equations defining
$X$, $x$, $Z$ are elements of $K$.
More precisely, we assume that we are given
$f_1,\ldots,f_m \in \Z[x_1,\ldots,x_n]$
such that $\Z[x_1,\ldots,x_n]/(f_1,\ldots,f_m)$ is a domain
with fraction field $K$,
and that elements of $K$ are presented as rational expressions
in the generators.

Actually, we show that the problem is undecidable 
even if $X$, $x$, $Z$ are base extensions of $\Q$-varieties.
In fact, we prove a strong form of Theorem~\ref{T:main}:

\begin{theorem}
\label{T:fixed variety}
There is a {\em fixed} nice $\Q$-variety $X$
and a fixed rational point $x$ on $X$
such that it is impossible to decide which 
nice $\Q$-subvarieties $Z$ of $X$
meet $\{\sigma x : \sigma \in \Aut X\}$.
\end{theorem}
\noindent
That is, there is no algorithm that takes $Z$ as input
and decides whether there exists an automorphism of $X$ mapping $x$ into $Z$.

Finally, our $X$ in Theorem~\ref{T:fixed variety}
will have $\Aut X = \Aut X_\C$,
where $X_\C$ is the base extension $X \underset{\Spec \Q}\times \Spec \C$,
so it does not matter whether we consider 
only automorphisms defined over $\Q$ or also automorphisms over $\C$.

These problems are proved undecidable by relating them to
Hilbert's tenth problem.
Hilbert asked for an algorithm to decide, given
a multivariable polynomial equation with integer coefficients,
whether or not it was solvable in integers.
But Matiyasevich~\cite{Matiyasevich1970}, 
building on earlier work of 
Davis, Putnam, and Robinson~\cite{Davis-Putnam-Robinson1961},
proved that no such algorithm exists.

\begin{remark}
If $X$ is a nice variety of general type, the problems above
are {\em decidable} because $\Aut X$ is finite and computable
as a subgroup of some $\PGL_n$ acting on some pluricanonical image of $X$.
\end{remark}

\begin{remark}
This is not the first time that a problem in algebraic geometry
has been proved undecidable.
The problem of deciding whether a rational map of complex varieties
$X \dashrightarrow \PP^2$ admits a rational section
is undecidable~\cite{Kim-Roush1992} 
(this is equivalent to the analogue of Hilbert's tenth problem 
for $\C(T_1,T_2)$).
The generalization with $\PP^2$ replaced by any fixed complex
variety of dimension at least $2$ is undecidable too~\cite{Eisentraeger2004}.
(But the analogue for $\PP^1$ is still open, as is the analogue
for any other fixed curve.)
\end{remark}

\begin{remark}
Burt Totaro asked the author in 2007 whether the problem of
deciding whether two varieties are isomorphic is undecidable.
\end{remark}

\section{Lattice automorphisms preserving a finite subset}

The group of affine linear automorphisms of $\Z^n$
is the semidirect product $\GL_n(\Z) \ltimes \Z^n$, 
with $(A,\vec{b})$ acting as $\vec{x} \mapsto A \vec{x} + \vec{b}$.

\begin{lemma}
\label{L:lattice}
For each $n \ge 3$, there exists a finite subset $S$ of $\Z^n$
containing $\vec{0}:=(0,0,\ldots,0)$
such that the subgroup of $\GL_n(\Z) \ltimes \Z^n$ mapping $S$ to $S$
equals the subgroup $G$ of linear maps given by matrices
\[
\begin{pmatrix}
  1 &   &        &   & a_1 \\
    & 1 &        &   & a_2 \\
    &   & \ddots &   & \vdots \\
    &   &        & 1 & a_{n-1} \\
    &   &        &   & a_n 
\end{pmatrix}
\]
with $a_i \in \Z$ for all $i$ and $a_n = \pm 1$.
\end{lemma}

\begin{proof}
Let $p_i$ be the $i^\tH$ prime.
For $1 \le i \le n-1$, let $v_i \in \Z^n$ be the vector
with $p_i$ in the $i^\tH$ coordinate and $0$ elsewhere.
Let $S=\{\vec{0},v_1,\ldots,v_{n-1}\}$.
Let $G'$ be the subgroup of $\GL_n(\Z) \ltimes \Z^n$ mapping $S$ to itself.
Suppose that $g \in G'$.
Then $g$ fixes $\vec{0}$ since each other vector in $S$
differs from some other vector by a primitive vector.
Also $g$ fixes $v_i$ for each $i$, 
since $v_i$ is distinguished from the other $v_j$ 
by being divisible by $p_i$.
So $g$ fixes $S$ pointwise.
It hence acts trivially on the real affine linear span of $S$,
so it acts trivially on $\Z^{n-1} \times 0$.
Thus $G' \subseteq G$.
Conversely, elements of $G$ map $S$ to $S$.
So $G'=G$.
\end{proof}

\section{Blowups of powers of an elliptic curve}
\label{S:construction}

In this section, we prove a weak version of Theorem~\ref{T:main}
in which $Z$ is not required to be smooth or integral.

Fix an elliptic curve $E$ over $\Q$ such that $\End E \isom \Z$
and such that $E(\Q)$ contains a point $P$ of infinite order.
For instance, one could take $E$ to the curve 
labelled 37A1 in~\cite{Cremona1997},
with equation $y^2+y=x^3-x$.
Let $n \ge 3$.
Let $X$ be the blowup of $E^n$ at the subset $S' \subset (\Z\cdot P)^n$
corresponding to the subset $S \subset \Z^n$ given by Lemma~\ref{L:lattice}.
For a variety $V$,
we write $\Aut V$ for the group of automorphisms of $V$ 
{\em as a variety without extra structure}, 
even if $V$ is an abelian variety.
The birational morphism $X \to E^n$ is the map from $X$ to its Albanese torsor,
so there is an injective homomorphism $\Aut X \to \Aut E^n$
whose image equals the subgroup of $\Aut E^n$ mapping $S'$ to itself.
Any such automorphism of $E^n$ must be of the form
$\vec{x} \mapsto A\vec{x} + \vec{b}$
for some $A \in \GL_n(\Z)$ and $\vec{b} \in E^n$,
but $S' \subset (\Z \cdot P)^n$ so $\vec{b} \in (\Z \cdot P)^n$.
It follows that $\Aut X$ is isomorphic to the group $G$ 
in Lemma~\ref{L:lattice}.
Identify the exceptional divisor $D$ above $\vec{0} \in E^n$
with $\PP^{n-1}$ in the natural way.
Let $x = (0:\cdots:0:1) \in \PP^{n-1} = D \subseteq X$.
If $\sigma \in \Aut X$ corresponds to 
\[
\begin{pmatrix}
  1 &   &        &   & a_1 \\
    & 1 &        &   & a_2 \\
    &   & \ddots &   & \vdots \\
    &   &        & 1 & a_{n-1} \\
    &   &        &   & a_n 
\end{pmatrix}
\in G,
\]
then $\sigma x=(a_1:\cdots:a_n) \in \PP^{n-1}$. 

Given a polynomial $f(t_1,\ldots,t_{n-1}) \in \Z[t_1,\ldots,t_{n-1}]$,
let $F(t_1,\ldots,t_n) \in \Z[t_1,\ldots,t_n]$
be its homogenization,
and let $Z$ be the zero locus of $F$ in $\PP^{n-1} = D \subseteq X$.
Then $f$ has a zero in $\Z^{n-1}$
if and only if $F$ has a zero in $\Z^{n-1} \times \{\pm 1\}$,
which holds if and only if $\sigma x \in Z$ for some $\sigma \in \Aut X$.

Since the general problem 
of deciding whether a polynomial in $\Z[t_1,\ldots,t_{n-1}]$
has an zero in $\Z^{n-1}$ is undecidable,
the general problem of deciding
whether $\sigma x \in Z$ for some $\sigma \in \Aut X$
is undecidable too.

\section{Making the subvariety smooth}

\begin{lemma}
\label{L:smooth affine}
There is an algorithm that, given a nonconstant $f \in \Z[x_1,\ldots,x_n]$,
constructs a polynomial $F \in \Z[x_1,\ldots,x_{n+1}]$ such that 
\begin{enumerate}
\item[(i)]
The equation $f(\vec{a})=0$ has a solution $\vec{a} \in \Z^n$
if and only if $F(\vec{b})=0$ has a solution $\vec{b} \in \Z^{n+1}$.
\item[(ii)]
The affine variety $X:=\Spec \Q[x_1,\ldots,x_{n+1}]/(F)$
is smooth and geometrically integral.
\item[(iii)]
We have $\deg F=2 \deg f$.
(Here $\deg$ denotes total degree.)
\end{enumerate}
\end{lemma}

\begin{proof}
Consider $F(x_1,\ldots,x_n,y)=c (y^2-y) + f(x_1,\ldots,x_n)^2$
for some $c \in \Z_{>0}$.
The values of $y^2-y$ and $f(x_1,\ldots,x_n)^2$ on integer inputs
are nonnegative, so (i) is satisfied.
The singular locus $S$ of $X$
is contained in the locus where $\del F/\del y=0$,
which is $2y-1=0$ in $\Aff^N$.
On the other hand, 
Bertini's theorem (\cite{Hartshorne1977}*{Remark~III.10.9.2}) 
shows that $S$ is contained in $y^2-y=0$ for all but finitely many $c$.
In this case $S=\emptyset$,
so $X$ is smooth over $\Q$.
By testing $c=1,2,\ldots$ in turn, we can effectively find 
the first $c$ for which $X$ is smooth over $\Q$.

This $X$ is also geometrically integral:
since $X$ is isomorphic to a variety of the form $z^2-g=0$
for some nonconstant $g \in \Z[x_1,\ldots,x_n]$,
if it were not geometrically integral, $z^2-g$ would factor
as $(z+h)(z-h)$ for some nonconstant $h \in \Qbar[x_1,\ldots,x_n]$,
but then $X$ would have to be singular at the common zeros of $z$ and $h$,
a contradiction.
\end{proof}

\begin{lemma}
\label{L:resolution}
There is an algorithm that, 
given an affine scheme $U$ of finite type over $\Z$
whose generic fiber $U_\Q$ is smooth over $\Q$,
constructs $n \in \Z_{>0}$ and a closed immersion $U \injects \Aff^n_\Z$ 
such that the projective closure of the generic fiber $U_\Q$ in $\PP^n_\Q$
is smooth.
\end{lemma}

\begin{proof}
(This proof grew out of a discussion with 
Andrew Kresch, Florian Pop, and Yuri Tschinkel.)
Embed $U$ as a closed subscheme of some $\Aff^m_\Z$,
and let $X$ be its closure in $\PP^m_\Z$.
Let $H=X-U$.
Resolution of singularities lets us construct a coherent sheaf of ideals
$\calI_\Q$ on $X_\Q$ with support contained in $H_\Q$
such that blowing up $X_\Q$ along $\calI_\Q$
yields a smooth $\Q$-scheme.
Let $\calI := \calI_\Q \intersect \OO_X$ in $\OO_{X_\Q}$,
so $\calI$ is a coherent sheaf of ideals on $X$
such that blowing up $X$ along $\calI$ yields $X' \stackrel{\pi}\to X$
with $X'_\Q$ smooth over $\Q$.
Let $E$ be the exceptional divisor on $X'$.
By \cite{EGA-II}*{8.1.8}, the Cartier divisor $-E$ 
(corresponding to the invertible sheaf $\calI \OO_{X'}$)
is ample relative to $X' \to X$.
By \cite{EGA-II}*{4.4.10(ii)},
for sufficiently large $d$, the divisor $D:=d \pi^*H - E$ on $X'$
is ample and has support equal to the closed subset $\pi^{-1}(H)$.
So a multiple of $D$ determines a closed immersion $X' \injects \PP^n_\Z$
such that $X' \intersect \Aff^n_\Z = \pi^{-1}(U) \isom U$.
All these constructions can be made effective.
\end{proof}

Combining the previous two lemmas with the negative solution
to Hilbert's tenth problem yields:

\begin{corollary}
\label{C:H10 for smooth varieties}
There is no algorithm that,
given $f_1,\ldots,f_m \in \Z[x_1,\ldots,x_n]$ 
such that 
the projective closure of $\Spec \Q[x_1,\ldots,x_n]/(f_1,\ldots,f_m)$
in $\PP^n_\Q$ is smooth and geometrically integral over $\Q$,
decides whether $f_1(\vec{a})=\cdots=f_m(\vec{a})=0$
has a solution $\vec{a} \in \Z^n$.
\end{corollary}

Applying the construction of Section~\ref{S:construction}
to the smooth projective geometrically integral $\Q$-variety $Z$
arising as the projective closure in 
Corollary~\ref{C:H10 for smooth varieties}
proves Theorem~\ref{T:main}.

\section{Uniformity} 

In this section we prove Theorem~\ref{T:fixed variety}.
In our proof of Theorem~\ref{T:main},
the variety $X$ and the point $x$
depend only on the integer $n$ chosen at the beginning of 
Section~\ref{S:construction}.

The negative solution of Hilbert's tenth problem
shows that there are fixed $m$ and $d$
such that the problem of deciding whether an $m$-variable
polynomial of total degree $d$ is solvable in natural numbers
is undecidable~\cite{Matiyasevich1970}.
Replacing each variable by a sum of squares of four new variables
and applying Lagrange's theorem that every nonnegative integer
is a sum of four squares
shows that the same uniform undecidability holds for solvability 
in integers, provided that we replace $(m,d)$ by $(4m,2d)$.
Combining this with Lemma~\ref{L:smooth affine} yields 
undecidability even if we restrict
to polynomials defining a smooth affine hypersurface over $\Q$,
provided that we replace $(m,d)$ by $(m+1,2d)$.
Because these smooth affine hypersurfaces form an algebraic family,
the resolution of singularities of each projective closure
has bounded complexity for all the hypersurfaces in the family,
so the proof of Lemma~\ref{L:resolution}
re-embeds these hypersurfaces
in a projective space of bounded dimension,
which can then be embedded in a larger projective space 
of {\em fixed} dimension $D$.
Finally we may take $n=D+1$ in Section~\ref{S:construction}.
This completes the proof of Theorem~\ref{T:fixed variety}.

\section*{Acknowledgements} 

It was a discussion with Burt Totaro that inspired the research
leading to this article.
I thank Andrew Kresch, Florian Pop, and Yuri Tschinkel 
for a teatime discussion 
at the Hausdorff Institute for Mathematics 
that led to the proof of Lemma~\ref{L:resolution}.
I thank also Laurent Moret-Bailly for a brief suggestion
regarding references.

\end{document}